\documentclass[11pt,a4paper]{article}
\usepackage{epsfig}
\usepackage[numbers]{natbib}
\usepackage{amsmath, amsthm, amssymb}
\usepackage{epsfig}
\usepackage[dvipdfm]{hyperref}
\usepackage[margin=3cm, top=3.5cm, bottom=3.5cm]{geometry}

\newcommand\blfootnote[1]{%
  \begingroup
  \renewcommand\thefootnote{}\footnote{#1}%
  \addtocounter{footnote}{-1}%
  \endgroup
}
\newtheorem{theorem}{Theorem}[section]
\newtheorem{lemma}{Lemma}[section]

\newtheorem{remark}{Remark}[section]

\numberwithin{equation}{section}
\everymath{\displaystyle}
\allowdisplaybreaks

\title{Improved Hardy inequality with logarithmic term}
\date{}
\author{
Nikolai Kutev\thanks{Institute of Mathematics and Informatics, Bulgarian Academy of Sciences, 1113, Sofia, Bulgaria}
 \and Tsviatko Rangelov
 \footnotemark[1]
}

\begin{document}

\maketitle
\blfootnote{Corresponding author: T. Rangelov, rangelov@math.bas.bg}

\begin{abstract}
\noindent
New Hardy type inequality with double singular kernel and with additional logarithmic term in  a ball $B\subset \mathbb{R}^n$ is proved. As an application an estimate from below of the first eigenvalue for Dirichlet problem  of p-Laplacian in a bounded domain $\Omega\subset \mathbb{R}^n$ is obtain.
\end{abstract}


{\bf Keywords} Hardy inequalities, First eigenvalue of p-Laplacian.

{\bf Math. Subj. Class.} 26D10, 35P15

\section{Introduction}
\label{sec1}
Classical Hardy inequalities in bounded domain $\Omega\subset \mathbb{R}^n$, $n\geq2$, $p>1$, $p\neq n$  have the form, see Sect. 2.1 in \cite{KR22} and references therein:
\begin{itemize}
\item  with singularity at $0\in\Omega$ and $p>n$
\begin{equation}
\label{eq001}
\int_\Omega|\nabla u|^pdy\geq \left(\frac{p-n}{p}\right)^p\int_\Omega\frac{|u|^p}{|y|^p}dy,\quad u\in W_0^{1,p}(\Omega\backslash\{0\}),
\end{equation}
\item  with singularity on $\partial\Omega$, with $d(y)=\hbox{dist}(y,\partial\Omega)$ and constant $C_\Omega$
\begin{equation}
\label{eq002}
\int_\Omega|\nabla u|^pdy\geq C_\Omega\int_\Omega\frac{|u|^p}{d^p(y)}dy,\quad u\in W_0^{1,p}(\Omega).
\end{equation}
\end{itemize}
There are different improvements of (\ref{eq001}) and (\ref{eq002}) among them are inequalities with double singular kernels in the ball $B_R=\{|y|<R\}\subset \mathbb{R}^n$, $n\geq2$, $\beta\in(1,n)$,   $p>n$ and with additional power term, see \cite{KR19a}
\begin{equation}
\label{eq003}
\begin{array}{ll}
\int_{B_R}|\nabla u|^pdy
&\geq\left(\frac{p-\beta}{p}\right)^p\int_{B_R}\frac{|u|^p}{|y|^{\frac{(\beta-1)p}{p-1}}\phi(|y|)^p}dy
\\
&+(n-\beta)\left(\frac{p-\beta}{p}\right)^{p-1}\int_{B_R}\frac{|u|^p}{|y|^{\beta}\phi(|y|)^{p-1}}dy,
\end{array}
\end{equation}
for $u\in W_0^{1,p}(B_R\backslash\{0\})$,
where
\begin{equation}
\label{eq004}
\phi(|y|)=\left(R^{\frac{p-\beta}{p-1}}-|y|^{\frac{p-\beta}{p-1}}\right).
\end{equation}
The aim of this paper is to prove one parametric family of new Hardy inequalities with double singular kernels and additional logarithmic term in the ball $B_R$. For this purpose we solve explicitly the Poisson problem for the p-Laplacian in $B_R$ with special right-hand side $h(|y|)$ for every $p>1$. By means of the solutions to the Poisson problem, when the right-hand side $h(|y|)$ is one parametric family of functions singular at the origin, we prove Hardy inequalities with kernels singular at the origin and at the boundary $\{|y|=R\}$. The additional logarithmic term is obtain from a nonlinear ordinary differential inequality of first order, see \cite{BFT03a}, \cite{BFT03b}, \cite{FKR15}.
The following theorem is proved:
\begin{theorem}
\label{th1}
If $n\geq2$, $p>n$, $\beta\in(1,n)$,  $b<-\frac{p-2}{6(p-1)}$, $x_0=\frac{1-\sqrt{1+4|b|}}{-2|b|}$, then for every $u\in W_0^{1,p}(B_R\backslash\{0\})$ Hardy type inequality
\begin{equation}
\label{eq005}
\begin{array}{ll}
\int_{B_R}|\nabla u|^pdy&\geq\left|\frac{p-\beta}{p}\right|^p\int_{B_R}\frac{|u|^p}{|y|^{\frac{(\beta-1)p}{p-1}}
\phi(|y|)^{p-1}}
\\[1pt]
&\times\left(1+\frac{p}{2(p-1)}\frac{1}{\left(1+\ln\frac{\phi(|y|)}{e^{1/x_0}}\right)^2}\right)dy
\\[1pt]
\\
&+\left|\frac{p-\beta}{p}\right|^{p-1}
(n-\beta)\int_{B_R}\frac{|u|^p}{|y|^{\beta}
\phi(|y|)^{p-1}}
\\[1pt]
\\
&\times\left(1-\frac{1}{\ln\frac{\phi(|y|)}{e^{1/x_0}}}-\frac{|b|}{\ln^2\frac{\phi(|y|)}{e^{1/x_0}}}\right)dy.
\end{array}
\end{equation}
holds.
\end{theorem}
Note that the inequality (\ref{eq005}) improves the inequality (\ref{eq003}).

As an application of the new Hardy inequalities (\ref{eq005}) we improve the analytical estimate from below of the first eigenvalue $\lambda_{p,n}(\Omega)$ for the p-Laplacian in a bounded smooth domain $\Omega\subset \mathbb{R}^n$, $n\geq2$,  $p>n$. Firstly, this estimate is obtained from the new Hardy inequalities in the ball $B_R$ and then by the Faber-Krahn inequality is extended to an arbitrary bounded smooth domain $\Omega$, see \cite{LW97}, \cite{BK02}.

The analytical formulae for the first eigenvalue $\lambda_{p,n}(\Omega)$ are obtained  only:
 \begin{itemize}
 \item for $n=1$ in \cite{Ot84}: $\lambda_{p,1}(-1,1)=(p-1)\left(\frac{\pi}{p\sin\frac{\pi}{p}}\right)^p$,
 \item for $n\geq2$ and  $p=2$, in special domains like a ball, spherical shell and parallelepiped, for example in a ball $B_R$, see \cite{Vl71}:  $\lambda_{2,n}(B_R)=\left(\frac{\mu_1^{(\alpha)}}{R}\right)^2, \
\ \ \alpha=\frac{n}{2}-1$, $\mu_1^{(\alpha)}$ is the first positive zero of the Bessel
function $J_{\alpha}$.
\end{itemize}
For arbitrary $p>2$ and $n\geq2$ there are different methods in order to estimate $\lambda_{p,n}(\Omega)$ from below, see \cite{KR22}. Some of them are by Cheeger`constant  \cite{KF03}, \cite{Ch70}, with Picone identity  \cite{BD12}, \cite{BD13}, by means of the Sobolev constant \cite{LXZ11}, \cite{Ma85}, estimates in a parallelepiped \cite{Li95} and by Hardy inequalities with double singular kernels \cite{KR22}, Chap. 8, for $n\geq2, p>1, p\neq n$
\begin{equation}
\label{eq006}
\lambda_{p,n}(B_R)\geq
\frac{1}{R^p}\left(\frac{1}{p}\right)^p\left[\frac{(p-1)^{p-1}}{(n-1)^{n-1}}\right]^{\frac{p}{p-n}}.
\end{equation}
The comparison of all these estimates from below of $\lambda_{p,n}(\Omega)$ for arbitrary $p>1$ and $n\geq2$ are shown in \cite{KR22}, Sect. 8.3. For small $p>1$,close to $1$ and $2\leq n\leq9$ better estimates are given by means of the Cheeger`s constant and the Picone identity. Unfortunately, the numerical calculations for $\lambda_{p,n}(\Omega)$ in \cite{BBEM12} are quite better then the analytical estimates. This motivates us to improve the well-known analytical estimates, especially (\ref{eq006})  by means of the new Hardy inequalities (\ref{eq005}) with double singular kernel and additional logarithmic term.

The plan of the paper is the following. In Sect. \ref{sec2} we prove some preliminary results: Lemma \ref{lem1} for the one parametric radial solution of the Poisson problem for p-Laplacian and Lemma \ref{lem2} for estimate of solutions of first order non-linear ordinary differential inequality. In Sect.\ref{sec3} using the results in Sect. \ref{sec2} and the method derived in Section 3.1 of \cite{KR22} we prove Theorem \ref{th1}. Finally, in Sect. \ref{sec4} is derived a new estimate of $\lambda_{p,n}$.

\section{Preliminary remarks}
\label{sec2}

For $n\geq2$, $p>1$ we consider the Poisson problem in $B_R$
\begin{equation}
\label{eq1} \left\{\begin{array}{l} -\hbox{div}(|\nabla
\phi|^{p-2}\nabla \phi)=h(|y|) \ \ \hbox{ in } B_R,
\\[1pt]
\\
\phi=0 \ \ \hbox{ on } \partial B_R,
\end{array}\right.
\end{equation}
where
\begin{equation}
\label{eq2}
\int_0^R s^{n-1}h(s)ds<\infty.
\end{equation}

We will use the following simple Lemma
\begin{lemma}[\cite{BBEM12}]
\label{lem1}
A solution of (\ref{eq1}),  (\ref{eq2}) is given by
\begin{equation}
\label{eq3}
\phi(|y|)=\int_{|y|}^R\theta^{\frac{1-n}{p-1}}\left(\int_0^\theta s^{n-1}h(s)ds\right)^{\frac{1}{p-1}}d\theta.
\end{equation}
\end{lemma}
\begin{proof}
From the invariance of the equation in (\ref{eq1})  under rotation, the solution of problem (\ref{eq1}) is radially symmetric  and satisfies the boundary value problem for ordinary differential equation
\begin{equation}
\label{eq4} \left\{\begin{array}{l}-(r^{n-1}|\phi'|^{p-2}\phi')'=r^{n-1}h(r), \ \ 0<r<R,
\\[1pt]
\\
\phi(R)=0.
\end{array}\right.
\end{equation}
Integrating twice the equation in (\ref{eq4}) and applying boundary condition we obtain (\ref{eq3}).
\end{proof}

In order to obtain new  Hardy inequality (\ref{eq005}) let us prove the following auxiliary lemma, see also \cite{BFT03a}, \cite{BFT03b}, \cite{FKR15a}.
\begin{lemma}
\label{lem2}
For every $p\geq2$, there exists $b=b(p)<0$ and $S_0>0$ such that for every $S>S_0$  the function
\begin{equation}
\label{eq5}
w(t)=\left(\frac{1}{p'}\right)^{p-1}\left(1-\frac{1}{1+\ln S-t}+\frac{b}{(1+\ln S-t)^2}\right),
\end{equation}
satisfies
\begin{equation}
\label{eq6}
w(t)\in C^1(-\infty,0), \ \ w>0, \ \ w'<0, \ \ w(-\infty)=\left(\frac{1}{p'}\right)^{p-1},
\end{equation}
and is  a solution of the inequality
\begin{equation}
\label{eq7}
-w'+(p-1)w-(p-1)w^{p'}\geq G(t),
\end{equation}
where
\begin{equation}
\label{eq007}
G(t)=\left(\frac{1}{p'}\right)^p\left(1+\frac{p}{2(p-1)}\frac{1}{(1+\ln S-t)^2}\right).
\end{equation}
\end{lemma}
\begin{proof}
Let us denote for simplicity $x(t)=\frac{1}{1+\ln S-t}$, so that
$$
w(t)=\left(\frac{1}{p'}\right)^{p-1}(1-x+bx^2), \ \ w'(t)=\left(\frac{1}{p'}\right)^{p-1}(-x^2+2bx^3)
$$
Expanding $w^{p'}(x)$ for small $x$ near $x=0$ in a Taylor polynomial up to the third order we obtain
$$
\begin{array}{lll}
w^{p'}&=&\left(\frac{1}{p'}\right)^p\left\{1-\frac{p}{p-1}x+\frac{p}{p-1}\left(2b+\frac{1}{p-1}\right)
\frac{x^2}{2}\right.
\\[1pt]
\\
&+&\left.\frac{p}{p-1}\left[-\frac{6b}{p-1}+\frac{p-2}{(p-1)^2}\right]\frac{x^3}{6}+o(x^3)\right\}
\end{array}
$$
Then if
$$
b<-\frac{p-2}{6(p-1)}
$$
we get
$$
\begin{array}{lll}
&&-w'+(p-1)w-(p-1)w^{p'}
\\[1pt]
\\
&&=\left(\frac{1}{p'}\right)^p\left[1+\frac{p}{2(p-1)}x^2+p\left(-b-\frac{p-2}{6(p-1)}\right)x^3+o(x^3)\right]
\\[1pt]
\\
&&\geq\left(\frac{1}{p'}\right)^p\left(1+\frac{p}{2(p-1)}x^2\right)
\end{array}
$$

With this choice of $b$ inequalities (\ref{eq7}) and $w'(s)<0$ hold. In order to satisfy the rest of  conditions (\ref{eq6}) we  choose $S$ such that $w(t)>0$, i.e.,  $1-x-|b|x^2>0$. This means that
$$
0<x<x_0=\frac{1-\sqrt{1+4|b|}}{-2|b|}.
$$
If $S_0$ is such that $S_0>e^{\frac{1}{x_0}-1}$ then for every $S>S_0$ we get $w(t)>0$.
\end{proof}

\section{Main result}
\label{sec3}
Let us define the vector function
$g=\frac{|\nabla\phi|^{p-2}\nabla\phi}{|\phi|^{p-2}\phi}w(\ln\phi)$ where $w$ is defined in (\ref{eq5}), Lemma \ref{lem2} and $\phi$ is defined in (\ref{eq004})

Simple calculations give us that $g\in C^1(B_R\backslash\{0\})$ satisfies the equality

 $$
\begin{array}{lll}
-\hbox{div}g&=&-\left(\frac{\Delta_p\phi}{|\phi|^{p-1}}-(p-1)
\left|\frac{\nabla\phi}{\phi}\right|^p\right)w(\ln\phi)-\left|\frac{\nabla\phi}{\phi}\right|^p
w'(\ln\phi)
\\[1pt]
\\
&=&\left|\frac{\nabla\phi}{\phi}\right|^p\left[-w'+(p-1)w-(p-1)w^{p'}\right]
+(p-1)\left|\frac{\nabla\phi}{\phi}\right|^pw^{p'}-\frac{\Delta_p\phi}{|\phi|^{p-1}}w,
\end{array}
$$
and the inequality
\begin{equation}
\label{eq101}
-\hbox{div}g\geq(p-1)|g|^{p'}+v,
\end{equation}
holds,
where $v=\left|\frac{\nabla\phi}{\phi}\right|^pG(\ln\phi)+\frac{h(|y|)}{|\phi|^{p-1}}w(\ln\phi)$, $v\in C^1(B_R\backslash\{0\})$. Here $G$ is defined in (\ref{eq007}), Lemma \ref{lem2} and function $h$ satisfies (\ref{eq2}).

By means of the notations
\begin{equation}
\label{eq9}
\begin{array}{l}
L(u)=\int_{B_R}\left|\frac{\langle\nabla\phi,
\nabla u\rangle}{|\nabla\phi|}\right|^pdy\geq \int_{B_R}|\nabla u|^pdy,
\\[1pt]
\\
K(u)=\displaystyle\int_{B_R}\left|\frac{\nabla \phi}{\phi}\right|^p|u|^pdy,
\ \ N(u)=\displaystyle\int_{B_R}v|u|^pdy,
\end{array}
\end{equation}
We will formulate the following Lemma
\begin{lemma}
\label{lem3}
If $n\geq2$, $p>1$, $p\neq n$ and  (\ref{eq101}) is satisfied then the following Hardy inequality holds
\begin{equation}
\label{eq10} L(u)\geq N(u), \ \ u\in C^{\infty}_0(B_R)
\end{equation}
\end{lemma}
\begin{proof}
 According  to Theorem 1 in \cite{FKR15} using the properties of $g$ we obtain
\begin{equation}
\label{eq11}
L(u)\geq\left(\displaystyle\frac{1}{p}\right)^p\frac{[(p-1)K(u)+N(u)]^p}{K^{p-1}(u)},
\ \ u\in C^{\infty}_0(B_R),
\end{equation}

 Applying  the Young inequality
$$
 \frac{P^p}{Q^{p-1}}\geq
pm^{p-1}P-(p-1)m^pQ
$$
with $Q>0$, $P\geq0$ and constant $m\geq0$ to the right-hand-side of
(\ref{eq11}) we get
\begin{equation}
\label{eq13} L(u)\geq(p-1)m^{p-1}(1-k)K(u)+ m^{p-1}N(u).
\end{equation}
In particular, for $m=1$ in (\ref{eq13}) we get a linear form of Hardy inequality (\ref{eq10}).
\end{proof}

\begin{proof}[of Theorem \ref{th1}]
We  choose a function $h(|y|)$  such that the function $\phi$ has a simple form and also the kernel  $N(u)$ has double singularity at $0$ and on $\partial B_R$.

Let us get the function $h(|y|)=\frac{p-\beta}{p-1}(n-\beta)|y|^{-\beta}$, $\beta\in(1,n)$, then (\ref{eq2}) holds and from (\ref{eq3}) we have $\phi(|y|)$ defined in (\ref{eq004}).

The expressions for  $N(u)$  in (\ref{eq9}) is
\begin{equation}
\label{eq15}
\begin{array}{ll}
N(u)&=\left(\frac{p-\beta}{p}\right)^p\int_{B_R}\frac{|u|^p}{|y|^{\frac{(\beta-1)p}{p-1}}
\left|\phi(|y|)\right|^{p-1}}
\\[1pt]
&\times\left(1+\frac{p}{2(p-1)}\frac{1}{\left(1+\ln\frac{\phi(|y|)}{e^{1/x_0}}\right)^2}\right)dy
\\[1pt]
\\
&+\left(\frac{p-\beta}{p}\right)^{p-1}
(n-\beta)\int_{B_R}\frac{|u|^p}{|y|^{\beta}
\left|\phi(|y|)\right|^{p-1}}
\\[1pt]
\\
&\times\left(1-\frac{1}{\ln\frac{\phi(|y|)}{e^{1/x_0}}}-\frac{|b|}{\ln^2\frac{\phi(|y|)}{e^{1/x_0}}}\right)dy.
\end{array}
\end{equation}
Note that condition $\beta\in(1,n)$ gives that $\phi(|y|)\leq1$ for $y\in B_R$. Since $e^{1/x_0}>1$ then $\ln\frac{\phi(|y|)}{e^{1/x_0}}<0$ and the kernels of the right-hand-side of (\ref{eq15}) are continuous for  $y\in B_R$.
\end{proof}

\begin{remark}\rm
\label{rem1}
The inequality (\ref{eq10}) with $N$ in (\ref{eq15}) is better then inequality (23) of Lemma 2 in \cite{KR19b} due to the logarithmic terms.
\end{remark}

\section{Estimates of the first eigenvalue of $p$- Laplacian in the ball}
\label{sec4}
As an application of the new Hardy inequality (\ref{eq10})
in this section we  estimate from below the first eigenvalue of
the p--Laplacian in a bounded smooth domain $\Omega\subset R^n$,  $n\geq2$,
$p>1$, i. e.
$$
\left\{\begin{array}{l} -\hbox{div}(|\nabla
u|^{p-2}\nabla u)=\lambda_{p,n}(\Omega)|u|^{p-2}u \ \ \hbox{ in } \Omega,
\\[1pt]
\\
u=0 \ \ \hbox{ on } \partial \Omega,
\end{array}\right.
$$
 The first eigenvalue
$\lambda_{p,n}(\Omega)$ is defined as
\begin{equation}
\label{eq17} \lambda_{p,n}(\Omega)=\inf_{u\in
W^{1,p}_0(\Omega)}\displaystyle\frac{\displaystyle\int_{\Omega}|\nabla
u|^pdx}{\displaystyle\int_{\Omega}|u|^pdx}.
\end{equation}
and $\lambda_{p,n}(\Omega)$ is simple, i.e., the first eigenfunction $\varphi(x)$
is unique up to multiplication with nonzero constant $C$. Moreover,
$\varphi$ is positive in $\Omega$, $\varphi\in W^{1,p}_0(\Omega)\cap
C^{1,s}(\bar{\Omega})$ for some $s\in(0,1)$, see e. g.
\cite{BK02} and the references therein.

Note that for an arbitrary bounded domain $\Omega\subset R^n$ the Faber--Krahn type inequality gives the estimate
\begin{equation}
\label{eq18} \lambda_{p,n}(\Omega)\geq\lambda_{p,n}(\Omega^{\ast})
\end{equation}
where $\Omega^{\ast}$ is the n--dimensional ball of the same volume as $\Omega$, see \cite{LW97}, \cite{BK02}.

 Thus from (\ref{eq18}) it is enough to prove lower bound of $\lambda_{p,n}$ only for a ball $B_R$.

For this purpose for $n\geq2$, $p>n$, $\beta\in(1,n)$ we define the function
$$
\begin{array}{lll}
&H(p,n,\beta,r)
\\[1pt]
\\
&=\left(\frac{p-\beta}{p}\right)^p\frac{1}{r^{\frac{(\beta-1)p}{p-1}}
\left|\phi(r)\right|^{p-1}}\left(1+\frac{p}{2(p-1)}\frac{1}{\left(1+\ln\frac{\phi(r)}{e^{1/x_0}}\right)^2}\right)
\\[1pt]
\\
&+\left(\frac{p-\beta}{p}\right)^{p-1}
(n-\beta)\frac{1}{r^{\beta}
\left|\phi(r)\right|^{p-1}}\left(1+\frac{1}{\ln\frac{\phi(r)}{e^{1/x_0}}}-\frac{|b|}{\ln^2\frac{\phi(r)}{e^{1/x_0}}}\right).
\end{array}
$$
\begin{theorem}
\label{th2}
For $n\geq 2$, $p>n$, $\beta\in(1,n)$  the estimate
\begin{equation}
\label{eq19}
\lambda_{p,n}(B_R)\geq \sup_{1<\beta<n}\inf_{r\in(0,R)}H(p,n,\beta,r)
\end{equation}
holds.
\end{theorem}
\begin{proof}
Replacing $|y|=r$ in (\ref{eq005}) and using (\ref{eq17}), we get (\ref{eq19}).
\end{proof}

Note that in the case $p>n$ the estimate (\ref{eq19}) is better than the estimate (\ref{eq006}) for $p>n$ and is better then the the estimate
$$
\lambda_{p,n}(B_R)\geq \inf_{r\in(0,R)}H(p,n,n,r)
$$
obtained in \cite{FKR15a}.

\textbf{Acknowledgement}
\small{The work is  partially supported by   the Grant No BG05M2OP001--1.001--0003, financed by the Science and Education for Smart Growth Operational Program (2014-2020) in Bulgaria and co-financed by the European Union through the European Structural and Investment Funds.}

\begin{flushleft}
Institute of Mathematics and  Informatics,\\ Bulgarian Academy
of Sciences \\ Acad. G. Bonchev str.,bl. 8\\
Sofia 1113, Bulgaria, \\ E-mail: kutev@math.bas.bg; rangelov@math.bas.bg,
\end{flushleft}

\end{document}